\documentclass[11pt]{article}
\usepackage[ansinew]{inputenc}
\usepackage[T1]{fontenc}
\usepackage{graphicx,textcomp,booktabs,stmaryrd}
\usepackage{mathptmx}

\usepackage{mathdots}
\usepackage[english]{babel} 
\usepackage[a4paper]{geometry} 
\usepackage{amssymb,amsmath,amsfonts,amsthm} 
\usepackage{enumerate} 
\usepackage{pict2e,pstricks,pst-plot,pst-node}
\usepackage{bbm} 
\usepackage{dsfont} 
\usepackage{wrapfig}

\newtheorem{prop}{Proposition}
\newtheorem{theorem}{Theorem}
\newtheorem{lemma}{Lemma}

\newtheorem*{ackno}{Acknowledgments}

\newcommand{\R}{\mathbb{R}}

\newcommand{\N}{\mathbb{N}}

\newcommand{\PP}{\mathbb{P}}
\newcommand{\E}{{\mathbb{E}}}

\newcommand{\defi}{\equiv} 

\newcommand{\de}{\delta}

\newcommand{\vare}{\varepsilon}

\newcommand{\beq}{\begin{equation}} 
\newcommand{\eeq}{\end{equation}} 
\newcommand{\bea}{\begin{aligned}}
\newcommand{\eea}{\end{aligned}}
\newcommand{\bdm}{\begin{displaymath}}
\newcommand{\edm}{\end{displaymath}}
\newcommand{\barr}{\begin{array}}
\newcommand{\earr}{\end{array}}
\newcommand{\ben}{\begin{enumerate}}
\newcommand{\een}{\end{enumerate}}
\newcommand{\bde}{\begin{description}}
\newcommand{\ede}{\end{description}}

\author{Marius A. Schmidt \\  J.W. Goethe-Universit\"at Frankfurt, Germany \\
mschmidt@math.uni-frankfurt.de}

\title{A SIMPLE PROOF OF THE DPRZ-THEOREM \\
FOR 2d COVER TIMES.}

\begin{document}
\maketitle
The {\it $\varepsilon$-cover time} of the two dimensional unit torus $\mathbb{T}_2$ by Brownian motion (BM) is the time for the process to come within distance  $\vare> 0$ from any point. Denoting by $T_{\varepsilon}\left(x\right)$ the first time BM hits the $\varepsilon$-ball centered in $x\in\mathbb{T}_2$, the $\vare$-cover time is thus given by \beq T_{\varepsilon} \equiv \sup_{x\in \mathbb{T}_2} T_{\varepsilon}\left(x\right). \eeq
The purpose of these short notes is to provide a concise proof of a celebrated theorem by Dembo, Peres, Rosen and Zeitouni, DPRZ for short, which settles the leading order in the small-$\vare$ regime:

\begin{theorem} (The DPRZ-Theorem, \cite{DPRZ}) \label{main_thrm} Almost surely, 
\beq
\lim_{\varepsilon \downarrow 0} \frac{T_{\varepsilon}}{\left(\ln \varepsilon \right)^2} = \frac{2}{\pi}.
\eeq
\end{theorem}


\noindent A key idea in the DPRZ-approach is to relate hitting times of $\vare$-balls on $\mathbb{T}_2$ to excursion counts between circles of mesoscopic sizes around these balls \cite{ray}; the DPRZ-proof of the theorem goes then through an involved multiscale analysis in the form of a second moment computation with truncation. We take here a similar point of view  but with a number of twists which altogether lead to a considerable streamlining of the arguments. In particular, we implement the multiscale refinement of the second moment method emerged in the recent studies of Derrida's GREM and branching Brownian motion \cite{kistler}. This tool brings to the fore the {\it true} process of covering \cite{belius_kistler} with the help of minimal infrastructure only; it also efficiently replaces the delicate tracking of points which DPRZ refer to as 'n-successful', and requires the use of finitely many scales only. All these features simplify substantially the proof of the DPRZ-theorem.

We believe the route taken here also considerably streamlines the deep DPRZ-results on late and thin/thick points of BM  \cite{dembo}, and, what is perhaps more, it will be useful in the study of the {\it finer} properties. In fact, our approach carries over, {\it mutatis mutandis}, to these issues as well: when backed with \cite{belius_kistler}, the present notes suggest that in order to address lower order corrections, one "simply" needs to increase the number of scales. 

These notes are self-contained. Although, as mentioned, some key insights are taken from \cite{DPRZ}, no knowledge of the latter is assumed and detailed proofs to all statements are given. 

\section{The (new) road to the DPRZ-Theorem}\label{proof_sec}
We identify the unit torus $\mathbbm{T}_2$ with $\left[0,1\right) \times \left[0,1\right) \subset \R^2$, endowed with the metric $$ d_{\mathbb{T}_2}\left(x,y\right)= \min\left\{||x-y+\left(e_1,e_2\right)||: e_1,e_2\in \{-1,0,1\} \right\}.$$ 
We construct BM on $\mathbbm{T}_2$ by $W_t \defi \left(\hat{W}_1(t) \mod 1, \hat{W}_2(t) \mod 1\right)$, where $\hat{W}$ is standard BM on $\R^2$.  \\

\noindent By monotonicity of $T_\varepsilon$ and Borel-Cantelli Lemma, the DPRZ-Theorem steadily follows from 

\begin{theorem} \label{one_t}
For $\delta>0$ small enough there exist constants $c\left(\delta\right),c'\left(\delta\right)>0$ such that the following bounds hold for any $0<\varepsilon<c'\left(\delta\right)$:
\begin{itemize}
\item[1)] (upper bound) 
\beq \PP \left(T_\varepsilon > \left(1+\delta\right)\frac{2}{\pi}\left(\ln\varepsilon\right)^2 \right)\leq \varepsilon^{c\left(\delta\right)}, \eeq
\item[2)] (lower bound) \beq \PP \left(T_\varepsilon < \left(1-\delta\right)\frac{2}{\pi}\left(\ln\varepsilon\right)^2 \right)\leq \varepsilon^{c\left(\delta\right)}\,. \eeq
\end{itemize}
\end{theorem}
Theorem \ref{one_t} will be proved by relating the natural timescale of the covering process to the excursion-counts of an embedded random walk, and a multiscale analysis of the latter which exploits some underlying, approximate hierarchical structure in the spirit of \cite{belius_kistler}.  

\subsection{Scales, embedded random walks and excursion-counts}
For $R\in \left(0,\frac{1}{2}\right)$ and $K\geq 1$ we consider {\it scales} $i=0,1,..,K$ and associate to each such scale a radius
\beq\label{scales_eq} r_i \defi R\left(\frac{\varepsilon}{R}\right)^{i/K}. \eeq
BM started on $\partial B_{r_i}$ hits $\partial B_{r_{i+1}}$ before $\partial B_{r_{i-1}}$ with probability $1/2$: by the strong Markovianity and rotational invariance, it follows that the process obtained by tracking the order in which BM visits the scales (with respect to one fixed center point and not counting multiple consecutive hits to the same scale) during one excursion from scale $1$ to scale $0$ is a simple random walk  (SRW) started at $1$ and stopped in $0$. Keeping track of all BM-excursions up to some time thus yields a collection of independent SRW-excursions from $1$ to $0$. 
\begin{figure}[h]
\begin{center}
    \includegraphics[width=0.35\textwidth]{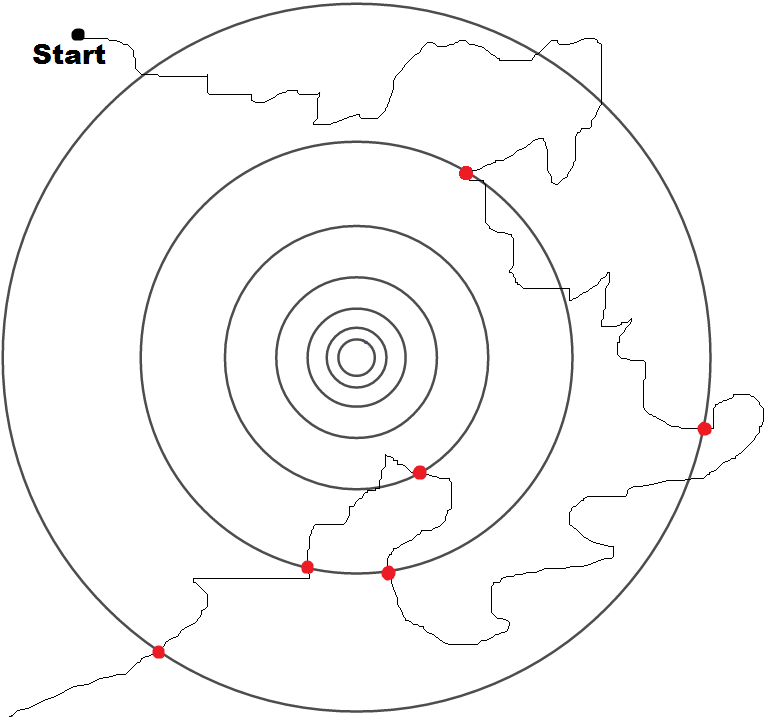}
\end{center}
\caption{Reading off the SRW excursions $1\rightarrow 0$ and $1\rightarrow 2\rightarrow 1\rightarrow 0$}
\label{bm_to_srw1}
\end{figure}
(The evolution of the SRW-excursions can be unambiguously read off the BM-path, see Figure \ref{bm_to_srw1}). For $x\in \mathbbm{T}_2$,  we set 
\beq D_n\left(x\right) \defi \mbox{time at which } W \mbox{ completes the } n\mbox{-th excursion from } \partial B_{r_1}\left(x\right) \mbox{ to } B^c_{r_0}\left(x\right). \eeq

\begin{prop} (Concentration of excursion-counts) \label{time_to_exc}
For $\delta,R \in \left(0,\frac{1}{2}\right)$ and $x\in \mathbbm{T}_2$, it holds
\beq \label{time_to_exc_eq1}\PP\left(D_n\left(x\right) \geq \left(1+\delta\right) n \frac{1}{\pi}\ln\frac{r_0}{r_1}\right) \leq \exp\left( -n \left(\frac{\delta^2}{8} +o_{r_1}(1)\right)\right)\eeq
\beq \label{time_to_exc_eq2}\PP\left(D_n(x) \leq  \left(1-\delta\right) n \frac{1}{\pi}\ln\frac{r_0}{r_1} \right) \leq \exp\left(-n \left(\frac{\delta^2}{4} +o_{r_1}\left(1\right)\right)\right)\eeq
for all $n\in \N$ as $r_1 \rightarrow 0$.
\end{prop}

\noindent Proposition \ref{time_to_exc} will bear fruits when combined with the following 
 
\begin{prop} (First moment of hitting times) \label{expectations}
There exists an universal constant $C>0$, such that
\beq \label{expec1}\left|\E_y[\tau_{B_r\left(x\right)}]-\frac{1}{\pi}\ln \frac{d_{\mathbbm{T}_2}\left(x,y\right)}{r}\right| \leq C \eeq
for all $x\in \mathbbm{T}_2$, $r>0$ and $y\in  \mathbbm{T}_2\setminus B_{r}\left(x\right)$. Also 
\beq \label{expec2}\E_y\left[\tau_{B^c_r\left(x\right)}\right] =  \frac{r^2 - d_{\mathbbm{T}_2}\left(x,y\right)^2}{2}\eeq
for all $x\in \mathbbm{T}_2$, $r\in \left(0,\frac{1}{2}\right)$ and $y\in B_{r}\left(x\right)$.
\end{prop}

\noindent Propositions \ref{time_to_exc} and \ref{expectations} make precise the intuition that $D_n(x) \approx n \E_{B_{r_0}}[\tau_{B_{r_1}}]$,  allowing in particular to switch from the natural timescale to the excursion-counts.  Armed with the above results, which will be proved in Section \ref{ex_sec}, we discuss the main steps behind Theorem \ref{one_t}.  The upper bound is easy: we address that first. \\

\noindent Here and below, $L_{\varepsilon}$ will denote the square lattice of mesh size $\lceil \varepsilon^{-1} \rceil ^{-1}$.  Remark that $|L_\vare| \approx \vare^{-2}$.  

\subsection{The upper bound}
We will show that, with overwhelming probability, at time 
\beq \label{shortening_t}
t_{\vare}(\delta) \defi \left(1+\delta\right) \frac{2}{\pi}\left(\ln \varepsilon\right)^2, 
\eeq 
each $\varepsilon$-ball with center on $L_\varepsilon$ has been hit by BM and extend this to the entire torus thereafter.

\begin{lemma}\label{ub_prep}
For $\delta>0$ small enough there exist constants $c, c' >0$ depending on $\delta$ only such that
\beq\PP\left(\exists x\in L_{\varepsilon}\mbox{ such that } T_\vare(x)>t_\vare(\delta) \right)\leq \varepsilon^{c}\eeq
holds for all $0<\varepsilon<c'$.
\end{lemma}
\begin{proof}
We set  
\beq n_\vare(\de) =-\left(1+\delta/2\right)2K\ln\left(\varepsilon\right), \eeq
which is slightly larger then the typical amount of excursions at time $t_{\vare}(\delta)$. For an $\varepsilon$-ball to be avoided up to some time: either {\it i)} BM needs to complete less than $n_\vare(\de)$ excursions from scale $1$ to scale $0$ in that time or {\it ii)} scale $K$, corresponding to the $\varepsilon$-ball, has to be avoided for at least $n_\vare(\de)$ many excursions. Therefore setting 
\[ 
\mathcal{T}\left(x\right) \defi \mbox{number of the first excursion from } \partial B_{r_1}\left(x\right) \mbox{ to } B^c_{r_0}\left(x\right) \mbox{ that hits } B_{r_K}\left(x\right).
\]
we have 
\beq \label{ub_eq1}
\PP\left(\exists x\in L_{\varepsilon} \mbox{ s.t. } T_\vare(x)>t_\vare(\delta) \right)\leq \PP\left(\exists x\in L_{\varepsilon} \mbox{ s.t. } \mathcal{T}\left(x\right) > n_\vare(\de) \mbox{ or } D_{n_\vare(\de)}\left(x\right) \geq  t_\vare(\delta)\right) .
\eeq
By Markov inequality and union bound
\beq \bea  \label{ah}
\eqref{ub_eq1} \leq\sum\limits_{x\in L_\varepsilon} \PP\left( \mathcal{T}\left(x\right) > n_\vare(\de) \right)+\PP\left(D_{n_\vare(\de)}\left(x\right) \geq  t_\vare(\delta)  \right)\,.
\eea\eeq
The probability that $n_\vare(\de)$ independent excursions of a SRW starting in $1$ all hit $0$ before $K$ is given by $(1-1/K)^{n_\vare(\de)}$, while the second probability on the r.h.s of \eqref{ah} is estimated by Proposition \ref{time_to_exc}. This shows that the above is at most
\beq  
\left|L_\varepsilon\right| \left[\left(1-\frac{1}{K}\right)^{n_\vare(\de)}+ \exp\left( -\frac{\delta^2}{72} n_\vare(\de) \right) \right] \leq \varepsilon^{\delta} \left(1+o_\varepsilon\left(1\right)\right)\,, 
\eeq
for $K$ large enough, the last inequality since $1-\frac{1}{K} \leq e^{-1/K}$, and $\left|L_\varepsilon\right| \approx \varepsilon^{-2}$.
\end{proof}

\noindent Coming back to the upper bound in Theorem \ref{one_t}, 
\beq \label{fif}
\PP \left( T_\varepsilon > t_\vare(\de) \right) = \PP\left(\exists x\in \mathbbm{T}_2: T_\varepsilon\left(x\right) > t_\vare(\de) \right) \leq \PP\left(\exists x\in L_{\varepsilon/10}: T_{\varepsilon/10}(x) > t_\vare(\de)  \right),
\eeq
the last step using that any $\varepsilon$-ball contains a ball of radius $\varepsilon/10$ with center in $L_{\varepsilon/10}$. For $\varepsilon>0$ small enough depending on $\delta$ we have $t_\vare(\de) \geq t_{\vare/10}(\delta/2)$, therefore it holds that
\beq \label{six}
 \eqref{fif} \leq\PP\left(\exists x\in L_{\varepsilon/10}: T_{\varepsilon/10}\left(x\right) > t_{\vare/10}(\delta/2) \right). 
\eeq
Applying Lemma $\ref{ub_prep}$ with $\varepsilon/10$ for $\varepsilon$ and $\delta/2$ yields the upper bound in Theorem \ref{one_t}.

\subsection{The lower bound}
\label{lb_sec}
We show that with overwhelming probability there exists $x\in \mathbbm{T}_2$ with avoided $\varepsilon$-ball at time 
\beq 
\mathfrak t = \mathfrak t(\vare, \delta) \defi  \left(1-\delta\right)^4\frac{2}{\pi}\left(\ln\varepsilon\right)^2\,.
\eeq 
Theorem \ref{one_t} will then follow immediately by considering\footnote{This is notationally convenient, but holds no deeper meaning.} $\hat{\delta} \defi 1- \left(1-\delta\right)^{4}$. Set 
\beq 
\mathfrak n(j) = \mathfrak n(j; \vare, \de, K) \defi -2K\left(1-\delta\right)^{j}\ln\varepsilon, \qquad (j\in \N).
\eeq
With $\tau_{r}\defi \tau_{r}\left(x\right)$ denoting the first time BM hits the $r$-ball around $x\in \mathbbm{T}_2$, we define the events
\beq \mathcal{R} \defi \bigcap_{x\in L_\varepsilon}\left\{D_{\mathfrak n(3)}(x) > \mathfrak t\right\} \quad \mbox{ and }
\eeq
\beq\bea  
\mathcal{R}^x \defi \left\{ \tau_{r_1} < \tau_{r_K} \right\} \cap  \left\{ \text{At most}\; \mathfrak n(2) \, \text{excursions}\; \lceil \delta k \rceil \shortrightarrow \lceil \delta k \rceil-1 \; \text{during first $\mathfrak n(3)$ excursions}\,  1\shortrightarrow 0  \right\}.
\eea \eeq
For $n\in \mathbb{N}$ and $l\in\left\{1,..,K-1\right\}$, let
\beq \bea 
\mathcal{N}_{l}^x\left(n\right) \defi & \; \mbox{number of excursions of } W \mbox{ from } \partial B_{r_l}\left(x\right) \mbox{ to } \partial B_{r_{l+1}}\left(x\right) \mbox{ within the } \\ &\mbox{ first } n \mbox{ excursions from } \partial B_{r_{l}}\left(x\right) \mbox{ to } \partial B_{r_{l-1}}\left(x\right) \mbox{ after time } \tau_{r_1}\mbox{}\,. 
\eea\eeq
For $x\in \mathbbm{T}_2$ define the events
\beq \bea  
A^x & \defi  \bigcap\limits_{l=\lceil \delta K\rceil}^{K-1} A_l^x, \quad \text{where}\quad
A^x_l & \defi  \left\{ \mathcal{N}^x_{l}\left(n \left(1-\frac{l}{K}\right)^2\right) \leq n  \left(1-\frac{l+1}{K}\right)^2 \right\}.
\eea\eeq
The events $A, \mathcal R$ are motivated by the following observations. First, it can be checked via Doob's h-transform that the expected number of excursions from $l$ to $l+1$ performed by a SRW started at $1$ and stopped at $0$ and conditioned not to hit $K$, is approximately  $\left[1-(l+1)/K\right]^2$. The events $A^x$ thus describe the natural avoidance strategy of scale $K$ by $n$ independent such SRW, which is in turn equivalent to specifying the avoidance strategy of an $\varepsilon$-ball.  Second, we claim that 
\beq \label{lowest}
\mathcal{R}\cap \mathcal{R}^x \cap A^x \subset \{ B_{\varepsilon}\left(x\right) \mbox{ is not hit up to time } t \}.\eeq

\noindent Remark in fact that on $\mathcal{R}^x$, the ball $B_\varepsilon\left(x\right)$ is not hit before $\partial B_{r_1}\left(x\right)$, hence the $\varepsilon$-ball can only be hit in an excursion from $B_{r_1}$ to $B_{r_0}$.  
$\mathcal{R} $ ensures that there are at most $\mathfrak n(3)$-excursions before \mbox{time $\mathfrak t$}. Therefore, on $\mathcal{R}^x\cap \mathcal{R}$, there are at most $\mathfrak n(2)$ excursions from scale $\lceil \delta K\rceil \rightarrow \lceil \delta K\rceil - 1$ at \mbox{time $\mathfrak t$}. But on $A^x$, none of these excursions reaches scale $K$, hence the $\varepsilon$-ball is not hit, and  \eqref{lowest} holds. 

In light of \eqref{lowest}, and in view of the lower bound in Theorem \ref{one_t},  estimates on the probabilities of the $\mathcal R, A$-events are needed. This information is provided by Lemma \ref{regular} and \ref{one_two_point} below, whose proofs are deferred to Section \ref{proofs_lemmata_lb}. Concerning the $\mathcal R$-event we state

\begin{lemma}\label{regular}
For all $\delta>0$ and large enough $K= K(\delta) \in \N$  there exist constants $\kappa, \kappa' >0$ 
depending  on $\delta, K$ only such that 
\beq 
\inf_{x\in L_\varepsilon \setminus B_{r_1}(W_0), \varepsilon \in (0,\kappa')} 
 \PP\left(\mathcal{R}^x \right),\PP(\mathcal{R}) \geq 1- \varepsilon^\kappa\,. 
\eeq
\end{lemma}
Concerning the $A$-events, 

\begin{lemma} \label{one_two_point} (One-point estimates) For $K$ large, $\varepsilon>0$ small enough (depending on $\delta$)
\beq  \label{one_point}
\varepsilon^{2- 1.99\delta} \leq \PP\left(A^x\right) \leq \varepsilon^{2- 2.01\delta},
\eeq
\end{lemma}

\noindent Coming back to the lower bound, restricting to the set \mbox{$L^*_\varepsilon \defi L_\varepsilon \setminus B_{r_1}(W_0)$} yields that
\beq \bea \label{lowboundstart} 
\PP\left( \sup_{x\in \mathbbm{T}_2} T_{\varepsilon}\left(x\right) > \mathfrak t  \right) & \geq  \PP\left(\exists x\in L_\varepsilon^* \mbox{ such that } B_{\varepsilon}(x) \mbox{ is not hit up to time } \mathfrak t\right) \\
& \stackrel{\eqref{lowest}}{\geq} \PP\left(\mathcal{R} \mbox{ and }\exists x\in L_\varepsilon^* \mbox{ such that } \mathcal{R}^x \cap A^x  \right) \\
& \geq \frac{\E\left[\#\left\{x\in L_\varepsilon^* :\mathcal{R}^x \cap A^x \right\}\right]^2}{\E\left[\#\left\{x\in L_\varepsilon^* : \mathcal{R}^x \cap A^x \right\}^2\right]} - \PP(\mathcal{R}^c), 
\eea \eeq
by Paley-Zygmund inequality. Rotational invariance and strong Markovianity imply that $ \mathcal{R}^x$ and $A^x$ are independent, hence the above is {\it at least} 
\beq \label{ue}
\left[\sum\limits_{x\in L_\varepsilon^*} \PP\left(\mathcal{R}^x \right) \PP\left( A^x\right) \right]^2 /\left[ \sum\limits_{x,y\in L^*_\varepsilon}  \PP\left(A^x \cap A^y\right)\right]-\PP(\mathcal{R}^c). 
\eeq
We now analyse the denominator. First, remark that for $d_{\mathbb T_2}(x,y) > 2 r_{\lceil \de K \rceil -1}$,  the $A$-events decouple: in fact, they are rotationally invariant and depend on disjoint excursions, hence
the strong Markov property yields $\PP \left( A^x  \cap A^y \right) = \PP \left( A^x  \right) \PP\left( A^y \right)$.
Shortening
\[ \mathcal A \defi \sum\limits_{x\in L_\varepsilon^*}  \PP\left( A^x\right), \quad 
\mathcal B \defi \sum\limits_{x,y\in L_\varepsilon} \mathbbm{1}_{\{d_{\mathbbm{T}_2}\left(x,y\right) \leq 2 r_{\lceil \delta K \rceil -1 }\}} \PP\left(A^x \cap A^y\right), 
\]  
by Lemma \ref{regular} and the exact decoupling we thus have that 
\beq \bea \label{almost_done}
\eqref{ue} & \geq  \left(1-\varepsilon^\kappa\right)^2 \frac{\mathcal A^2}{ \mathcal A^2+ \mathcal B}-\varepsilon^\kappa \geq \left(1-\varepsilon^\kappa\right)^2 \left(1 -\frac{\mathcal B}{\mathcal A^2} \right) - \vare^\kappa\\
& \geq \left(1-\varepsilon^\kappa\right)^2 \left(1 -\frac{\mathcal B}{\varepsilon^{-3.96\delta}} \right) - \vare^\kappa,
\eea \eeq
the last step by Lemma \ref{one_two_point} and using that $\left|L_\varepsilon\right| \geq \varepsilon^{-2+0.01\delta}$. It thus remains to analyze the $\mathcal B$-term: by regrouping terms according to the distance, 
\beq \bea \label{blabla}
\mathcal B & \leq  \sum\limits_{i = \lceil \delta K\rceil -2}^K \;\sum\limits_{x,y\in L_\varepsilon} \mathbbm{1}_{\{d_{\mathbbm{T}_2}\left(x,y\right) \in \left[r_{i+1},r_i\right]\}} \PP\left(A^x \cap A^y\right).
\eea \eeq
To get a handle on the two-points probabilities appearing in \eqref{blabla}, we follow the recipe from \cite[Sec. 3.1.1  p. 97-98]{kistler}, exploiting the approximate hierarchical structure which underlies the excursion-counts, and which is best explained with the help of a picture, see Figure \ref{decoup_fig} below. First, the circles associated to $x,y$ on small scales $i$ (left) are almost identical and so are the excursion counts; this suggests that $A^x_i \cap A^y_i$ is well represented by $A^x_i$ alone. Dropping one of the events is an estimate by worst case scenario known in this context as "REM approximation". For larger $i$ (middle) this approximation is not sharp, but only a single scale can fall into this case as we can choose $\varepsilon$ arbitrarily small for given $K$. Choosing $K$ large makes the influence of a single scale comparatively small. For $i$ large (right), balls are disjoint, which by rotational invariance and strong Markovianity yields independent excursion counts. Such approximate tree-structure of excursion counts is summarized in the lower picture, with the red box corresponding to the scale at hand. 
\begin{figure}
\begin{center}
    \includegraphics[width=\textwidth]{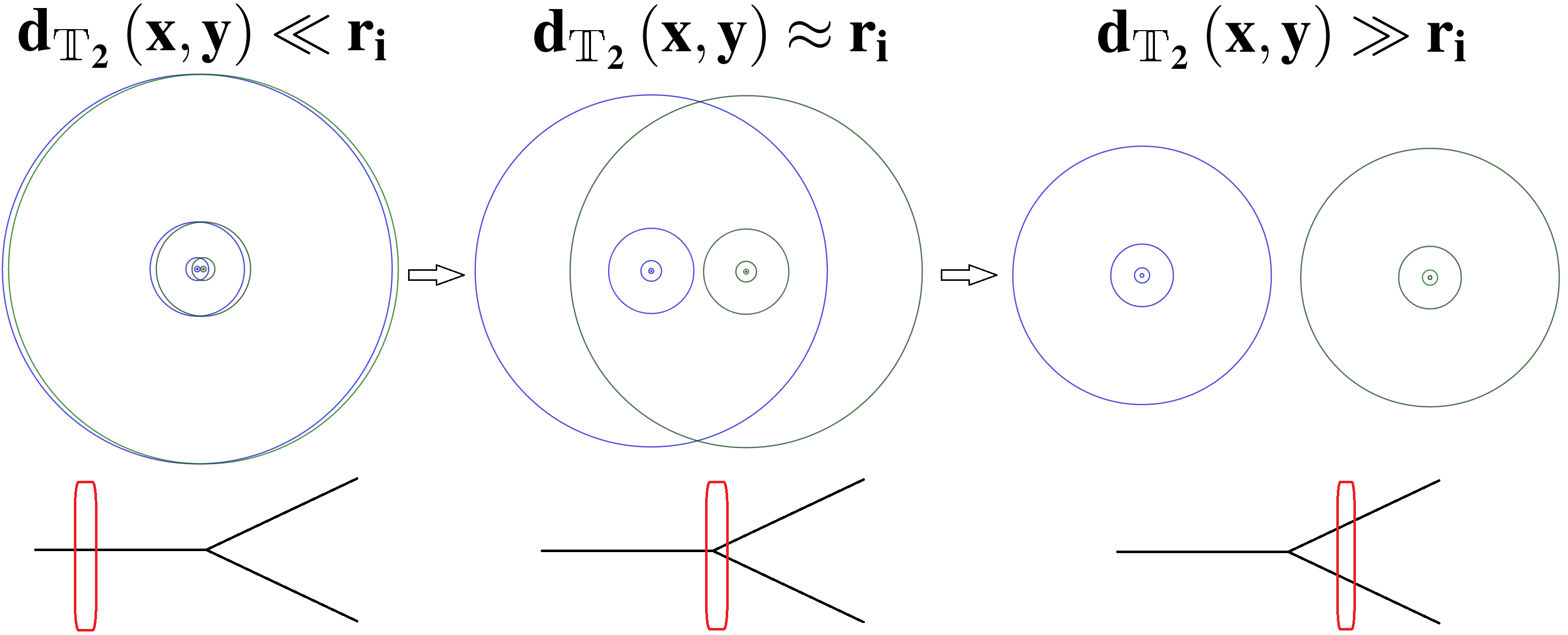}
\end{center}
\caption{Common branch on small scales (left) and decoupling on large scales (right).}
\label{decoup_fig}
\end{figure}
By these considerations, for $i \geq \lceil \delta K \rceil -2$ and $d_{\mathbbm{T}_2}\left(x,y\right) \in \left[r_{i+1},r_i\right]$, we write
\beq \bea \label{rem_approx_eq}
\PP\left(A^x \cap A^y\right) & = \PP\left( \bigcap\limits_{l=\lceil \delta K\rceil}^{K-1} A_l^x \cap  \bigcap\limits_{l=\lceil \delta K\rceil}^{K-1} A_l^y \right) \\
& \leq \PP\left( \bigcap\limits_{l=\lceil \delta K\rceil,l\neq i,i+1}^{K-1} A_l^x \cap  \bigcap\limits_{l= i+1}^{K-1} A_l^y \right) \qquad \text{("REM approximation")}\\
& = \prod\limits_{l=\lceil \delta K\rceil,l\neq i,i+1}^{K-1}\PP\left(  A_l^x \right)  \prod\limits_{l= i+1}^{K-1} \PP\left( A_l^y \right) \qquad \text{(exact decoupling)}\\
& \leq   \quad   \vare^{4-2.01 \de - 2 \frac{i+1}{K}}   \qquad \qquad \text{(Lemma \ref{one_two_point} / one-point estimates)} \,.
\eea \eeq
There are at most $2\varepsilon^{-4} \pi r_i^2$ pairs of points on $L_\varepsilon$ with distance at most $r_i$: using that $r_i \leq \varepsilon^{i/K}$, and \eqref{rem_approx_eq}  in \eqref{blabla} we get 
\beq 
\mathcal B \leq  \sum\limits_{i = \lceil \delta K\rceil -2}^K  2\pi \varepsilon^{-2.01\delta - \frac{4}{K}}\leq  \varepsilon^{-2.02\delta}\,.
\eeq
Applying this estimate to \eqref{almost_done} and putting $\hat{\delta} \defi 1- \left(1-\delta\right)^{4}$ we therefore see that
\beq \PP\left( \sup_{x\in \mathbbm{T}_2} T_{\varepsilon}\left(x\right) > \left(1-\hat\delta\right)\frac{2}{\pi}\left(\ln\varepsilon\right)^2 \right) \geq 1-
\vare^{\hat c},
\eeq
for  $\hat c \defi \frac{1}{2} \min\{ \kappa, 1.94\delta\}$, settling the lower bound of Theorem \ref{one_t}. 

\section{Proofs} 
\subsection{Hitting times and excursion-counts} \label{ex_sec}
The study of hitting times for BM is closely related to Green's functions. Estimates on the torus have however proofs which are either opaque or hard to find: we include here an elementary treatment based on Fourier analysis for the reader's convenience. 

\begin{lemma}\label{greenfunk}
The function 
\beq F\left(x,y\right) \defi G_x\left(y\right)- \frac{1}{2\pi} \ln d_{\mathbbm{T}_2}\left(x,y\right), \quad \text{where}\quad 
G_x\left(y\right) \defi -\sum\limits_{p\in 2\pi \mathbbm{Z}^2\setminus\{0\}}\frac{1}{\left|p\right|^2} e^{ip\left(x-y\right)}
\eeq
is bounded on $\mathbbm{T}_2^2 \setminus \{(x,x) : x\in \mathbbm{T}_2\}$.
\end{lemma}
\begin{proof}
It suffices to consider $y$ in a small neighborhood of $x$, as otherwise the result is trivial. So let $z \defi x-y$ and assume that $2|z_1| \geq |z|$ (swapping coordinates otherwise). We have 
\beq  \left|\sum\limits_{\substack{p\in 2\pi \mathbbm{Z}^2\setminus\{0\}  \\ |p|>|z|^{-1}}}\frac{1}{\left|p\right|^2} e^{ipz} \right| =   \left|\sum\limits_{\substack{p\in 2\pi \mathbbm{Z}^2\setminus\{0\}  \\ |p|>|z|^{-1}}}\frac{1}{1-e^{i 2\pi z_1}}\frac{1}{\left|p\right|^2} \left(e^{ipz}-e^{i(p+( 2\pi,0))z}\right) \right|. \eeq
Shifting the difference from the exponential to $|p|^{-2}$ by collecting terms with the same exponent, and by the triangle inequality, one obtains boundedness uniformly over $z\neq 0$ in a small enough neighborhood of $0$. The extra terms due to the boundary of the summation domain are easily shown to be bounded.  By combining the summand $p$ and $-p$ we see that sums of this form are real valued. Therefore 
\beq \sum\limits_{\substack{p\in 2\pi \mathbbm{Z}^2\setminus\{0\}  \\ |p| \leq |z|^{-1}}} \frac{1}{\left|p\right|^2} e^{ipz} = \sum\limits_{\substack{p\in 2\pi \mathbbm{Z}^2\setminus\{0\}  \\ |p|\leq |z|^{-1}}}\frac{1}{\left|p\right|^2} \cos(pz). \eeq
Since $|pz|\leq 1$ for all summands contained in this sum we can estimate $\cos(x) \leq 1- x^2/4$. Hence 
 \beq \left| G_x\left(y\right)-  \sum\limits_{\substack{p\in 2\pi \mathbbm{Z}^2\setminus\{0\}  \\ |p|\leq |z|^{-1}}}\frac{1}{\left|p\right|^2}\right|  \eeq
is uniformly bounded for $y$ in a small neighborhood of $x$. The claim of Lemma \ref{greenfunk} then follows by rearranging summands into groups $C_j \defi \{p\in 2\pi \mathbbm{Z}^2\setminus\{0\}: |p|^2 \in \left((j-1)^3,j^3\right]\}$, estimating $|p|^{-2}$ by best/worst case scenario within each group, and using that $|C_j| = \frac{3}{4\pi} j^2 + O(j^{3/2})$.
\end{proof}

\begin{proof}[Proof of Proposition \ref{expectations}: first moment of hitting times.]
Let $\mu(y)\defi \E_y\left[\tau_{B_r\left(x\right)}\right]$. 
For $\Delta$ the Laplacian with periodic boundary condition on $\mathbbm{T}_2$ we have Poisson's equation $\Delta\mu = -2$ on $\mathbbm{T}_2\setminus B_{r}\left(x\right)$ with $\mu = 0$ on $B_{r}\left(x\right)$. 
Plainly,
\beq  G_x\left(y\right) \defi -\sum\limits_{p\in 2\pi \mathbbm{Z}^2\setminus\{0\}}\frac{1}{\left|p\right|^2} e^{ip\left(x-y\right)}\eeq
is a Green function, i.e. solution of $\Delta G_x = 1 - \delta_x$ on the torus. In particular, $\mu +2G_x$ is harmonic on $\mathbbm{T}_2\setminus B_{r}\left(x\right)$. By the maximum principle, and since $\mu\defi 0$  on $\partial B_{r}\left(x\right)$,  
\beq 2 \inf\limits_{z\in \partial B_{r}\left(x\right)}G_x\left(z\right) \leq \mu\left(y\right) +2G_x\left(y\right) \leq 2\sup\limits_{z\in \partial B_{r}\left(x\right)}G_x\left(z\right)\eeq
holds. It follows from Lemma \ref{greenfunk} that $\mu\left(y\right) - \frac{1}{\pi}\ln[ d_{\mathbbm{T}_2}\left(x,y\right)/r]$ is bounded, and the first claim \eqref{expec1} is proved. The second claim \eqref{expec2} is elementary as we can identify the ball on $\mathbbm{T}_2$ with the ball in $\R^2$ and exploit rotational invariance to solve Poisson's equation explicitly.
\end{proof}

\begin{proof}[Proof of Proposition \ref{time_to_exc}: concentration of excursion-counts.]
By Kac's moment formula \cite{Kac}, 
\beq 
\E_x\left[\tau_A^i\right]\leq i! \sup\limits_{x\in \mathbbm{T}} \E_x\left[\tau_A \right]^i , \qquad A\subset\mathbbm{T} \; \text{closed}. 
\eeq 
By monotone convergence, Taylor-expanding the exponential function, and by the above estimate,
\beq \label{pos_exp_mom}  \E_x\left[e^{\theta \tau_A} \right] \leq 1 + \theta  \E_x\left[\tau_A\right] + \sum\limits_{i=2}^{\infty} \left(\theta  \sup\limits_{x\in \mathbbm{T}} \E_x\left[\tau_A \right] \right)^i \leq \exp\left(\theta  \E_x\left[\tau_A\right] +2\theta^2 \sup\limits_{x\in \mathbbm{T}} \E_x\left[\tau_A \right]^2  \right)\eeq 
for $0< \theta <\frac{1}{2} \left(\sup\limits_{x\in \mathbbm{T}} \E_x\left[\tau_A \right]\right)^{-1}$. Using $e^{-x} \leq 1-x+x^2$ for positive $x$ gives
\beq \label{neg_exp_mom} \E_x\left[e^{-\theta \tau_A} \right] \leq 1 - \theta  \E_x\left[\tau_A\right] +  \theta^2  \sup\limits_{x\in \mathbbm{T}} \E_x\left[\tau_A \right]^2 \leq \exp\left( - \theta  \E_x\left[\tau_A\right] +  \theta^2  \sup\limits_{x\in \mathbbm{T}} \E_x\left[\tau_A \right]^2\right).\eeq 
Consider $\tau^{\left(i\leftarrow\right)}$ the time it takes $W$ to get from $ \partial B_{r_1}\left(x\right)$ to $ B^c_{r_0}\left(x\right)$ the $i-th$ time; $\tau^{i\rightarrow}$ the time $W$ needs to get from  $ \partial B_{r_0}\left(x\right)$ to $ B_{r_1}\left(x\right)$ the $i$-th time after $ B_{r_1}\left(x\right)$ has been hit the first time and $\tau_{r_1}$ the time it takes $W$ to get from the starting point to $ \partial B_{r_1}\left(x\right)$. Now by definition we have
\beq\label{decomp_extime} D_n\left(x\right) = \tau_{r_1}+\sum\limits_{i=1}^{n-1} \tau^{\left(i\rightarrow\right)} +\sum\limits_{i=1}^{n} \tau^{\left(i\leftarrow\right)}.\eeq
Exponential Markov inequality gives for any $t,\theta >0$ 
\beq \PP\left(D_n\left(x\right) \geq t \right) \leq e^{-\theta t}\E\left[e^{\theta D_n\left(x\right)}\right]  \eeq
Using \eqref{decomp_extime}, by strong Markovianity and estimating by worst starting points this is 
\beq \leq \label{estim_expo_mom}  e^{-\theta t}\left(\sup\limits_{z\in\mathbbm{T}_2}\E_z\left[e^{\theta \tau_{r_1}}\right]\right)\left(\sup\limits_{z\in B_{r_0}\left(x\right)}\E_z\left[e^{\theta \tau^{\left(1\rightarrow\right)}}\right]\right)^{n-1}\left(\sup\limits_{z\in B_{r_1}\left(x\right)}\E_z\left[e^{\theta \tau^{(1\leftarrow)}}\right]\right)^{n} \eeq
Using \eqref{pos_exp_mom} with $\theta = -\frac{\pi \delta}{4\ln r_1 }$, and applying Proposition \ref{expectations}, we obtain
\beq \bea 
 &\;\;\sup\limits_{z\in\mathbbm{T}_2}\E_z\left[e^{\theta \tau_{r_1}}\right] &\leq&\; e^{\frac{\delta}{4} +\frac{\delta^2}{8}+ o_{r_1}\left(1\right)} \\
 &\sup\limits_{z\in B_{r_0}\left(x\right)}\E_z\left[e^{\theta \tau^{\left(1\rightarrow\right)}}\right]^{n-1}&\leq& \; e^{\left(n-1\right)\left(\frac{\delta}{4} +\frac{\delta^2}{8}+ o_{r_1}\left(1\right)\right)}\\
\mbox{and} &\sup\limits_{z\in B_{r_1}\left(x\right)}\E_z\left[e^{\theta \tau^{\left(1\leftarrow\right)}}\right]^n&\leq& \; e^{n o_{r_1}\left(1\right)}. \eea \eeq
With $t= \left(1+\delta\right) n \frac{1}{\pi}\ln\frac{r_0}{r_1}$, and by the above estimates, \eqref{estim_expo_mom} reads
\beq  \PP\left(D_n\left(x\right) \geq \left(1+\delta\right) n \frac{1}{\pi}\ln\frac{r_0}{r_1} \right) \leq e^{-n\left(\frac{\delta}{4} + \frac{\delta^2}{4}+o_{r_1}\left(1\right)\right)}e^{n\left(\frac{\delta}{4} + \frac{\delta^2}{8}+ o_{r_1}\left(1\right)\right)},\eeq
settling \eqref{time_to_exc_eq1}. As for \eqref{time_to_exc_eq2}: for any $n\in \N$ and $\theta>0$ we have
\beq \PP\left(D_n\left(x\right) \leq t \right) \leq e^{\theta t} \E e^{-\theta D_{n}\left(x\right)} \leq e^{\theta t} \E\left[ e^{-\theta  \tau^{\left(1\rightarrow\right)} }\right]^{n-1} . \eeq
Choosing $\theta = \frac{\pi \delta}{2 \ln r_1}$ and $t=   \left(1-\delta\right) n \frac{1}{\pi}\ln\frac{r_0}{r_1}$, applying \eqref{neg_exp_mom} together with Proposition \ref{expectations} yields the second claim and concludes the proof of Proposition \ref{time_to_exc}. 
\end{proof}

\subsection{Estimates for $\mathcal R$ and $A$}  \label{proofs_lemmata_lb}

\begin{proof}[Proof of Lemma \ref{regular}]
For $x\in L^*_\varepsilon$, $\left\{ \tau_{r_1} < \tau_{r_K} \right\}$ almost surely. By rotational invariance and strong Markovianity, the number of excursions from  scale $\lceil\delta K \rceil$ to scale $\lceil\delta K \rceil -1$ in different excursions from scale $1$ to scale $0$ are independent of each other. 
The number of excursions from scale $\lceil\delta K \rceil$ to scale $\lceil\delta K \rceil -1$ in one excursion from scale $1$ to scale $0$ is distributed like the product of a Bernoulli distributed and an independent geometrically distributed random variable, both with parameter $\lceil\delta K \rceil^{-1}$. (This product has expectation $1$). By Cram\'{e}r's theorem, 
\beq\bea 
& \PP\left( \mbox{more than } \mathfrak n(2) \mbox{ times } \lceil \delta K \rceil \rightarrow \lceil \delta K \rceil-1  \mbox{ in the first} \; \mathfrak n(3) \mbox{ excursions } 1\rightarrow 0  \right) \\
& \qquad \qquad \leq \exp\left(- \mathfrak n(3) I\left(\frac{1}{1-\delta}\right)\right) = \varepsilon^{2K\left(1-\delta\right)^{3} I\left(\frac{1}{1-\delta}\right)},\eea\eeq
with $I$ the rate function of a Bernoulli($1/ \lceil \delta K \rceil$ )$\times$ geometric($1/\lceil\delta K \rceil$).
It follows that $\PP\left(\left(\mathcal{R}^x\right)^c\right)$ vanishes polynomially in $\varepsilon$ for fixed $\delta$ and $K$. Taking the complement yields the first claim. 

By Proposition \ref{time_to_exc} we have
\beq \PP\left(D_{\mathfrak n(3)}\left(x\right) \leq \mathfrak t \right) \leq \varepsilon^{2K\left(1-\delta\right)^3 \left(\delta^2/4+o_{r_1}\left(1\right)\right)},\eeq 
which vanishes faster then, say, $\varepsilon^3$ for $K$ sufficiently large. The second claim thus follows by union bound over all $x\in L_\varepsilon$ on the complements. 
\end{proof}

\begin{proof}[Proof of Lemma \ref{one_two_point}]
The number of times a SRW goes from $l$ to $l+1$ before going from $l$ to $l-1$ is geo$(1/2)$-distributed. Therefore $\mathcal{N}^x_{l}\left(n\right)$ is, by strong Markovianity and rotational invariance, the sum of $n$ independent geo($1/2$)-distributed r.v.'s.  Hence by Cram\'{e}r's theorem 
\beq \bea\PP\left(A^x\right) = \prod\limits_{l=\lceil\delta K\rceil}^{K-1}\PP\left( A^x_l \right) &= &\prod\limits_{l=\lceil\delta K\rceil}^{K-1}\exp\left(-n \left(1-\frac{l}{K}\right)^2 I\left( \frac{  \left(1-\frac{l+1}{K}\right)^2}{ \left(1-\frac{l}{K}\right)^2}\right)+o_{\varepsilon}\left(n\right)\right)\\  &=& \exp\left(-\frac{n}{K^2}\sum_{l=\lceil\delta K\rceil}^{K-1} \left(K-l\right)^2 I\left(\left(1-\frac{1}{K-l}\right)^2\right)+o_{\varepsilon}\left(n\right)\right),\eea\eeq 
where $I\left(x\right)= x\ln\left(x\right) - \left(1+x\right)\ln\left(\frac{1+x}{2}\right)$ is the geo($1/2$)-rate function. Using $I\left(1\right) = I'\left(1\right) = 0$ and $I''\left(1\right)=\frac{1}{2}$ one quickly obtains $j^2 I\left(\left(1-1/j\right)^2\right) = 1 + o_j\left(1\right)$ as $j\rightarrow \infty$, and therefore
\beq \label{A_computations}\PP\left( A^x\right)  =\exp\left(-\frac{n}{K}\left(1-\delta\right)\left(1+ o_{K}\left(1\right)\right)+o_{\varepsilon}\left(n\right)\right) =\varepsilon^{2\left(1-\delta\right)\left(1+ o_{K}\left(1\right)\right) + o_\varepsilon\left(1\right)}, 
\eeq
concluding the proof of the Lemma. 
\end{proof}

\ackno{I am grateful to  David Belius, Giuseppe Genovese, Nicola Kistler, Benjamin Schlein and Tobias Weth for enlightening conversations.}


\begin{thebibliography}{99}
\bibitem{belius_kistler} D. Belius, and N. Kistler, \emph{The subleading order of two dimensional cover times}, Probability  Theory and Related Fields (2016)
\bibitem{dembo} A. Dembo, \emph{ Simple random covering, disconnection, late and favorite points.} Proceedings of the International Congress of Mathematicians, Madrid. Vol. 3. (2006).
\bibitem{DPRZ} A. Dembo, Y. Peres, J. Rosen, and O. Zeitouni, \emph{Cover times for Brownian motion and random walks in two dimensions}. Annals of Mathematics (2) Vol. 160 (2004): 433-464. 
\bibitem{Kac} P.J. Fitzsimmons, and J. Pitman, \emph{Kac's moment formula and the Feynman-Kac formula
for additive functionals of a Markov process}, Stochastic Processes and their Applications, Vol. 79 (1999): 
117-134.
\bibitem{kistler} N. Kistler, \emph{Derrida's Random Energy Models: From Spin Glasses to the extremes of correlated random fields}, in  \emph{Correlated random systems. Five different methods.}, V. Gayrard and N. Kistler (eds), Springer Lecture Notes in Mathematics, Vol. 2143 (2015)
\bibitem{ray} D. Ray, \emph{Sojourn times and the exact Hausdorff measure of the sample path for planar Brownian motion.} Transactions of the American Mathematical Society 106 (1963): 436-444.
\end{thebibliography}
\end{document}